\documentclass[12pt,reqno]{amsart}
\usepackage{amsmath}
\usepackage[dvips]{graphicx}
\usepackage{amsfonts}
\usepackage{amssymb}
\usepackage{latexsym}
\graphicspath{{Img/}}
\newtheorem{theorem}{Theorem}
\theoremstyle{plain}

\newtheorem{corollary}{Corollary}

\newtheorem{lemma}{Lemma}

\newtheorem{problem}{Problem}
\newtheorem{proposition}{Proposition}
\newtheorem{remark}{Remark}

\numberwithin{equation}{section}

\begin{document}

\title[Analytical solution of the weighted Fermat-Torricelli problem]{An analytical solution of the weighted Fermat-Torricelli problem on the unit sphere}
\author{Anastasios N. Zachos}
\address{University of Patras, Department of Mathematics, GR-26500 Rion, Greece}
\email{azachos@gmail.com} 
\begin{abstract}
We obtain an analytical solution for the weighted
Fermat-Torricelli problem for an equilateral geodesic triangle
$\triangle A_{1}A_{2}A_{3}$ which is composed by three equal
geodesic arcs (sides) of length $\frac{\pi}{2}$ for given three
positive unequal weights that correspond to the three vertices on
a unit sphere. This analytical solution is a generalization of
Cockayne's solution given in \cite{Coc:72} for three equal
weights. Furthermore, by applying the geometric plasticity
principle and the spherical cosine law, we derive a necessary
condition for the weighted Fermat-Torricelli point in the form of
three transcedental equations with respect to the length of the
geodesic arcs $A_{1}A_{1}^{\prime},$ $A_{2}A_{2}^{\prime}$ and
$A_{3}A_{3}^{\prime},$ to locate the weighted Fermat-Torricelli
point $A_{0}$ at the interior of a geodesic triangle $\triangle
A_{1}^{\prime}A_{2}^{\prime}A_{3}^{\prime}$ on a unit sphere with
sides less than $\frac{\pi}{2}.$
\end{abstract}\maketitle
\section{Introduction}

Let $\triangle A_{1}A_{2}A_{3}$ be a geodesic triangle and $A_{0}$
a point on a unit sphere.

We denote by $a_{ij}$ the length of the geodesic arc $A_iA_j,$
which is part of a great circle of unit radius and $\alpha_{ikj}$
the angle between the geodesic arcs $A_iA_k$ and $A_kA_j$ for
$i,j,k=0,1,2,3, i\neq j\neq k.$

The  weighted Fermat problem on the unit sphere refers to the
following problem:

\begin{problem}
Consider a positive constant weight $w_i$ that correspond to the
vertex $A_{i},$ for $i=1,2,3.$ Find a point $A_{0}$ (weighted
Fermat point) for which the sum
\begin{equation} \label{eq:obj}
\sum_{i=1}^{3}w_ia_{0i}
\end{equation}
is minimized.
\end{problem}


The existence and uniqueness for the weighted Fermat point on a
convex surface has been studied in \cite{Zach/Zou:08},
\cite{Fletcher:09}, \cite{Cots/Zach:11}, \cite{Zachos/Cots:10}
(see also in \cite{IvanTuzh:092}
\cite[Chapter~II,pp.~208]{BolMa/So:99}).

Concerning some studies that focus on the geometric properties of
the weighted Fermat point on the two dimensional sphere and on a
convex surface we refer to the studies of \cite{Coc:72},
\cite{Dolanetal:91}, \cite{Weng}, \cite{Coc:67},
\cite{IvanTuzh:092}, \cite{NayaInnami:2013} and \cite{Zachos:13b}.

The following results (Proposition 1, 2) characterize the
solutions of the weighted Fermat problem on a $C^{2}$ surface and
they have been proved in \cite{Zachos/Cots:10},\cite{Cots/Zach:11}
proposition 6, page 53 and proposition 7, page 55:

\begin{proposition}[Floating Case]{\cite[Proposition~6, p.~53]{Zachos/Cots:10},\cite{Cots/Zach:11}}
If $\vec{U}_{A_{i}A_{j}}$ is the unit tangent vector of the
geodesic arc $A_{i}A_{j}$ at $A_{i}$ and D is the domain of a
$C^{2}$ surface M bounded by $\triangle A_{1}A_{2}A_{3},$ for
$i,j=1,2,3$
 then the following (I), (II), (III) conditions are equivalent:
(I) All the following inequalities are satisfied simultaneously:
\begin{equation}\label{cond120n}
\left\|
w_{2}\vec{U}_{A_{1}A_{2}}+w_{3}\vec{U}_{A_{1}A_{3}}\right\|>
w_{1},
\end{equation}

\begin{equation}\label{cond1202n}
\left\|
w_{1}\vec{U}_{A_{2}A_{1}}+w_{3}\vec{U}_{A_{2}A_{3}}\right\|>
w_{2},
\end{equation}

\begin{equation}\label{cond1203n}
\left\|
w_{1}\vec{U}_{A_{3}A_{1}}+w_{2}\vec{U}_{A_{3}A_{2}}\right\|>
w_{3},
\end{equation}
(II) The point $A_{0}$ is an interior point of $\triangle
A_{1}A_{2}A_{3}$ (weighted Fermat-Torricelli point) and does not
belong to the geodesic arcs $A_{1}A_{2},$ $A_{2}A_{3}$
and $A_{1}A_{3}.$\\

(III) $\vec{U}_{A_{0}A_{1}}+\vec{U}_{A_{0}A_{2}}+\vec{U}_{A_{0}A_{3}}=\vec{0}.$\\

\end{proposition}
\begin{proposition}[Absorbed Case]{\cite[Proposition~7, p.~55]{Zachos/Cots:10},\cite{Cots/Zach:11}} The following (I), (II)
conditions are equivalent.\\(I) One of the following inequalities
is satisfied:
\begin{equation}\label{cond120}
\left\|
w_{2}\vec{U}_{A_{1}A_{2}}+w_{3}\vec{U}_{A_{1}A_{3}}\right\|\le
w_{1},
\end{equation}
or
\begin{equation}\label{cond1202}
\left\|
w_{1}\vec{U}_{A_{2}A_{1}}+w_{A_{3}}\vec{U}_{A_{2}A_{3}}\right\|\le
w_{2},
\end{equation}
or
\begin{equation}\label{cond1203}
\left\|
w_{1}\vec{U}_{A_{3}A_{1}}+w_{B}\vec{U}_{A_{3}A_{2}}\right\|\le
w_{3}.
\end{equation}

(II) The point $A_{0}$ (weighted Fermat-Cavalieri point) is
attained at $A_{1}$ or $A_{2}$ or $A_{3},$ respectively.
\end{proposition}

We note that there is no analytical solution with respect to the
weighted Fermat-Torricelli problem on the unit sphere, except of
Cockayne's solution given in \cite{Coc:72} for an equilateral
geodesic triangle having sides with length $\frac{pi}{2}$ for
three equal weights.

In this paper, we find an analytical solution of the weighted
Fermat-Torricelli problem for an equilateral geodesic triangle on
a unit sphere which is composed by three equal geodesic arcs of
length $\frac{\pi}{2},$ by using as variables the two angles of
longitude and latitude from the spherical coordinates and by
applying the spherical sine law in some specific geodesic
triangles (Theorem~\ref{theor2}). The geometric plasticity
principle which has been proved in \cite{Zachos:13b} yields a
class of geodesic triangles such that the corresponding weighted
Fermat-Torricelli point remains the same. By applying the
geometric plasticity principle, we find a class of geodesic
triangles by using the cosine law on the unit sphere, such that
the weighted Fermat-Torricelli point is the same with the weighted
Fermat-Torricelli point which corresponds to the equilateral
geodesic triangle. Finally, by applying the geometric plasticity
principle and the spherical cosine law, we derive a necessary
condition for the weighted Fermat-Torricelli point in the form of
three transcedental equations with respect to some specific three
length of geodesic arcs to locate the weighted Fermat-Torricelli
point $A_{0}$ at the interior of a geodesic triangle $\triangle
A_{1}^{\prime}A_{2}^{\prime}A_{3}^{\prime}$ on a unit sphere with
sides less than $\frac{\pi}{2}$
(Proposition~\ref{computeweightedFermatTorricelli}).


\section{Analytical solution of the weighted Fermat-Torricelli problem on the unit sphere}

Let $\triangle A_{1}A_{2}A_{3}$ be a geodesic triangle on the unit
sphere $S:x^2+y^2+z^2=1.$ such that
$a_{12}=a_{23}=a_{31}=\frac{\pi}{2}$ and $A_{1}=(1,0,0),$
$A_{2}=(0,1,0),$ $A_{3}=(0,0,1).$

\begin{figure} \label{fig:tas}
\centering
\includegraphics[scale=0.80]{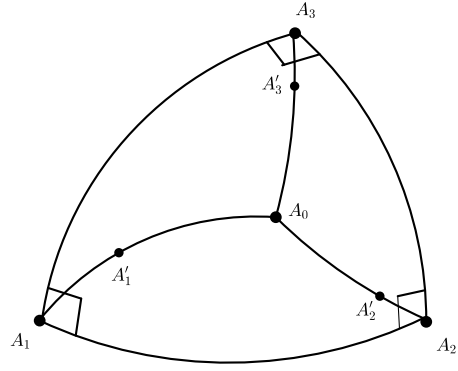}
\caption{Location of the weighted Fermat-Torricelli point for a
geodesic triangle on the unit sphere}
\end{figure}

\begin{lemma}{\cite[Theorem~1]{Zach/Zou:08},\cite{Zachos:13}}\label{theor1}
If $A_{0}$ is the weighted Fermat-Torricelli point of $\triangle
A_{1}A_{2}A_{3},$ then each angle $\alpha_{i0j}$ is expressed as a
function of $w_{1},$ $w_{2}$ and $w_{3}:$
\begin{equation}\label{cosi0j}
\alpha_{i0j}=\arccos\left(\frac{w_{k}^2-w_{i}^2-w_{j}^2}{2w_{i}w_{j}}\right)
\end{equation}
for $i,j,k=1,2,3,$ and $k\ne i\ne j.$
\end{lemma}

We start by expressing the position of the weighted
Fermat-Torricelli point $A_{0}=(x,y,z)$ in terms of the spherical
coordinates $(\omega,\varphi):$

$A_{0}=(\cos\omega cos\varphi, \cos\omega sin\varphi,\sin\omega).$

\begin{theorem}\label{theor2}
The analytical solution of the weighted Fermat-Torricelli problem
of $\triangle A_{1}A_{2}A_{3}$ on the unit sphere is given by the
following two relations:

\begin{equation}\label{explicit1}
\varphi=\arccos\left(\sqrt\frac{w_{1}^2+w_{3}^2-w_{2}^2}{2w_{3}^{2}}\right)
\end{equation}

and

\begin{equation}\label{explicit2}
\omega=\arccos\left(\sqrt{\frac{w_{1}^2+w_{2}^2-w_{3}^2}{2w_{1}w_{2}\sin\left(\arccos\left(\frac{w_{3}^2-w_{1}^2-w_{2}^2}{2w_{1}w_{2}}\right)\right)\sin\left(\arccos\left(\frac{w_{2}^2-w_{1}^2-w_{3}^2}{2w_{1}w_{3}}\right)\right)}}
\right),
\end{equation}

which yield the exact location of the weighted Fermat-Torricelli
point $A_{0}.$

\end{theorem}

\begin{proof}
The location of $A_{0}=(\cos\omega cos\varphi, \cos\omega
sin\varphi,\sin\omega).$ is determined by $\omega$ and $\varphi.$

We proceed by calculating $\omega$ and $\varphi$ with respect to
the given positive weights $w_{1},$ $w_{2}$ and $w_{3}.$

By applying the sine law in $\triangle A_{1}A_{0}A_{3},$
$\triangle A_{1}A_{0}A_{2}$ and $\triangle A_{2}A_{0}A_{3},$ we
get, respectively:

\begin{equation}\label{sinlaw1}
\frac{1}{\sin \alpha_{103}}=\frac{\sin a_{03}}{\sin
\alpha_{013}}=\frac{\sin a_{01}}{\sin \alpha_{130}},
\end{equation}

\begin{equation}\label{sinlaw2}
\frac{1}{\sin \alpha_{102}}=\frac{\sin a_{02}}{\sin
(\frac{\pi}{2}-\alpha_{013})}=\frac{\sin a_{01}}{\sin
\alpha_{120}}
\end{equation}

and

\begin{equation}\label{sinlaw3}
\frac{1}{\sin \alpha_{203}}=\frac{\sin a_{03}}{\sin
(\frac{\pi}{2}-\alpha_{120})}=\frac{\sin a_{02}}{\sin
(\frac{\pi}{2}-\alpha_{130})}.
\end{equation}


By taking the orthogonal projection of $A_{0}$ with respect to the
$xy$ plane and by using the Euclidean sine law, we express
$a_{01},$ $a_{02}$ and $a_{03}$ as functions of $\omega$ and
$\varphi:$

\begin{equation}\label{impsphere1}
\cos a_{03}=\sin \omega,
\end{equation}

\begin{equation}\label{impsphere2}
\cos a_{01}=\cos \omega \cos\varphi
\end{equation}

and

\begin{equation}\label{impsphere3}
\cos a_{02}=\cos \omega \sin\varphi.
\end{equation}


By replacing (\ref{impsphere1}), (\ref{impsphere2}) and
(\ref{impsphere3}) in (\ref{sinlaw1}), (\ref{sinlaw2}) and
(\ref{sinlaw3}) we obtain:

\begin{equation}\label{sinlaw1bis}
\frac{1}{\sin \alpha_{103}}=\frac{\cos \omega}{\sin
\alpha_{013}}=\frac{\sqrt{1-\cos^{2}\omega \cos^{2}\varphi}}{\sin
\alpha_{130}},
\end{equation}

\begin{equation}\label{sinlaw2bis}
\frac{1}{\sin \alpha_{102}}=\frac{\sqrt{1-\cos^{2}\omega
\sin^{2}\varphi}}{\cos \alpha_{013}}=\frac{\sqrt{1-\cos^{2}\omega
\cos^{2}\varphi}}{\sin \alpha_{120}}
\end{equation}

and

\begin{equation}\label{sinlaw3bis}
\frac{1}{\sin \alpha_{203}}=\frac{\cos \omega}{\cos
\alpha_{120}}=\frac{\sqrt{1-\cos^{2}\omega \sin^{2}\varphi}}{\cos
\alpha_{130}}.
\end{equation}

From (\ref{sinlaw1bis}) we get:

\begin{equation}\label{sin013}
\sin \alpha_{013}=\sin \alpha_{103}\cos \omega
\end{equation}

and

\begin{equation}\label{sin130}
\sin \alpha_{130}=\sin \alpha_{103}\sqrt{1-\cos^{2}\omega
\cos^{2}\varphi}.
\end{equation}

From (\ref{sinlaw2bis}) we get:

\begin{equation}\label{cos013}
\cos \alpha_{013}=\sin \alpha_{102}\sqrt{1-\cos^{2}\omega
\sin^{2}\varphi}
\end{equation}

and

\begin{equation}\label{sin120}
\sin \alpha_{120}=\sin \alpha_{102}\sqrt{1-\cos^{2}\omega
\cos^{2}\varphi}.
\end{equation}

From (\ref{sinlaw3bis}) we get:

\begin{equation}\label{cos120}
\cos \alpha_{120}=\sin \alpha_{203}\cos \omega
\end{equation}

and

\begin{equation}\label{cos130}
\cos \alpha_{130}=\sin \alpha_{203}\sqrt{1-\cos^{2}\omega
\sin^{2}\varphi}.
\end{equation}

By squaring both parts of (\ref{sin013}) and (\ref{cos013}) and by
adding the two derived equations we obtain:

\begin{equation}\label{fundamental1}
\sin^{2}\alpha_{103}\cos^{2}\omega+\sin^{2}\alpha_{102}(1-\cos^{2}\omega\sin^{2}\varphi)=1.
\end{equation}

By squaring both parts of (\ref{sin130}) and (\ref{cos130}) and by
adding the two derived equations we obtain:

\begin{equation}\label{fundamental2}
\sin^{2}\alpha_{203}\cos^{2}\omega+\sin^{2}\alpha_{102}(1-\cos^{2}\omega\cos^{2}\varphi)=1.
\end{equation}

By subtracting (\ref{fundamental2}) from (\ref{fundamental1}) and
taking into account (\ref{cosi0j}) from lemma~1 we derive:

\begin{equation}\label{fundamental3}
\cos2\varphi=\frac{B_{1}^2-B_{2}^2}{B_{3}^2}
\end{equation}

which yields (\ref{explicit1}).

By adding (\ref{fundamental2}) from (\ref{fundamental1}) and
taking into account the trigonometric identity
\[\sin^{2}\alpha_{102}=\sin^{2}\alpha_{203}\cos^{2}\alpha_{103}+\cos^{2}\alpha_{203}\sin^{2}\alpha_{103}+2\sin\alpha_{203}\cos\alpha_{203}\sin\alpha_{103}\cos\alpha_{103}\]
and (\ref{cosi0j}) from lemma~1 we derive (\ref{explicit2}).

\end{proof}

\begin{corollary}
If $w_{1}=w_{2}=w_{3},$ then $\omega=\frac{\pi}{4}$ and
$\varphi=\arccos{\sqrt{\frac{2}{3}}}$ and
$A_{0}=\left(\frac{1}{\sqrt{3}},\frac{1}{\sqrt{3}},\frac{1}{\sqrt{3}}\right)$
\end{corollary}

\begin{proof}
By replacing $w_{1}=w_{2}=w_{3}$ in (\ref{explicit1}) and
(\ref{explicit2}) we derive that $\omega=\frac{\pi}{4}$ and
$\varphi=\arccos{\sqrt{\frac{2}{3}}}$ and we deduce the position
of the Fermat-Torricelli point
$A_{0}=\left(\frac{1}{\sqrt{3}},\frac{1}{\sqrt{3}},\frac{1}{\sqrt{3}}\right).$
\end{proof}


The geometric plasticity principle of quadrilaterals on a convex
surface M, which is also valid on the unit sphere, states that
(\cite{Zachos:13b}):

\begin{proposition}[Geometric plasticity Principle]{\cite[Theorem~3, Proposition~8]{Zachos:13b}}\label{geomplasticity}
Suppose that the weighted floating case of the weighted Fermat
point $A_{0}$ point with respect to $A_{1}A_{2}A_{3}A_{4}$ is
valid:
\[\left\|
w_{Q}\vec{U}_{RQ}+w_{S}\vec{U}_{RS}+w_{P}\vec{U}_{RP}\right\|>
w_{R},\] for each $R,Q,S,P\in\{A_{1},A_{2},A_{3},A_{4}\}.$ If
$A_{0}$ is connected with every vertex $R$ for
$R\in\{A_{1},A_{2},A_{3},A_{4}\}$ and we select a point
$R^{\prime}$ with non-negative weight $w_{R}$ which lies on the
shortest arc $RA_{0}$ and the quadrilateral
$A_{1}^{\prime}A_{2}^{\prime}A_{3}^{\prime}A_{4}^{\prime}$ is
constructed such that:
\[\left\|
w_{Q}\vec{U}_{R^{\prime}Q^{\prime}}+w_{S}\vec{U}_{R^{\prime}S^{\prime}}+w_{P}\vec{U}_{R^{\prime}P^{\prime}}\right\|>
w_{R^{\prime}},\] for
$,R^{\prime},Q^{\prime},S^{\prime},P^{\prime}\in\{A_{1}^{\prime},A_{2}^{\prime},A_{3}^{\prime},A_{4}^{\prime}\}.$
 Then the weighted Fermat-Torricelli point
$A_{0}^{\prime}$ is identical with $A_{0}.$
\end{proposition}

\begin{lemma}\label{imp2}
The geometric plasticity principle holds for a geodesic triangle
on the unit sphere.
\end{lemma}

\begin{proof}
By replacing $w_{4}=0$ in Proposition~\ref{geomplasticity}, we
deduce the geometric plasticity principle of a geodesic triangle
$\triangle A_{1}A_{2}A_{3}$ on the unit sphere.
\end{proof}


Let $\triangle A_{1}^{\prime}A_{2}^{\prime}A_{3}^{\prime}$ be a
geodesic triangle on the unit sphere , such that $A_{i}^{\prime}$
belongs to the geodesic arc $A_{0}A_{i},$ for $i=1,2,3,$ where
$A_{0}$ is the weighted Fermat-Torricelli point of $\triangle
A_{1}A_{2}A_{3}$ (Fig.~1).

We assume that $a_{12}^{\prime},a_{23}^{\prime},a_{31}^{\prime}\le
\frac{\pi}{2},$ in order to locate the geodesic triangle
$\triangle A_{1}^{\prime}A_{2}^{\prime}A_{3}^{\prime}$ at the
interior of $\triangle A_{1}A_{2}A_{3}.$

Furthermore, we assume that the same weight $w_{i}$ that
corresponds to the vertex $A_{i}$ corresponds to the vertex
$A_{i}^{\prime},$ for $i=1,2,3,$ such that the inequalities of the
weighted floating case hold (Proposition~1).

We denote by $A_{0}^{\prime}$ the corresponding weighted
Fermat-Torricelli point of $\triangle
A_{1}^{\prime}A_{2}^{\prime}A_{3}^{\prime}.$

We denote by $a$ the length of the geodesic arc
$A_{1}A_{1}^{\prime},$ by $b$ the length of the geodesic arc
$A_{2}A_{2}^{\prime}$ and by $c$ the length of the geodesic arc
$A_{3}A_{3}^{\prime}.$

\begin{proposition}\label{computeweightedFermatTorricelli}
The following system of three equations with respect to $a,$ $b$
and $c$ provide a necessary condition to locate the weighted
Fermat-Torricelli point $A_{0}\equiv A_{0}^{\prime}$ at the
interior of a geodesic triangle $\triangle
A_{1}^{\prime}A_{2}^{\prime}A_{3}^{\prime}$ on a unit sphere with
sides less than $\frac{\pi}{2}:$

\begin{equation}\label{constructa}
\cos a_{12}^{\prime}= \cos(a_{01}-a) \cos(a_{02}-b)+\sin(a_{01}-a)
\sin(a_{02}-b)\frac{w_{3}^2-w_{1}^2-w_{2}^2}{2w_{1}w_{2}},
\end{equation}

\begin{equation}\label{constructb}
\cos a_{23}^{\prime}= \cos(a_{02}-b) \cos(a_{03}-c)+\sin(a_{02}-b)
\sin(a_{03}-c)\frac{w_{1}^2-w_{2}^2-w_{3}^2}{2w_{2}w_{3}},
\end{equation}

and

\begin{equation}\label{constructc}
\cos a_{13}^{\prime}= \cos(a_{01}-a) \cos(a_{03}-c)+\sin(a_{01}-a)
\sin(a_{03}-c)\frac{w_{2}^2-w_{1}^2-w_{3}^2}{2w_{1}w_{3}}.
\end{equation}

\end{proposition}

\begin{proof}

From lemma~2, the geometric plasticity holds on the units sphere.
Therefore, $A_{0}=A_{0}^{\prime}.$ By applying the cosine law in
$\triangle A_{1}^{\prime}A_{0}A_{2}^{\prime},$ $\triangle
A_{2}^{\prime}A_{0}A_{3}^{\prime}$ and $\triangle
A_{1}^{\prime}A_{0}A_{3}^{\prime},$ we obtain (\ref{constructa}),
(\ref{constructb}) and (\ref{constructc}), respectively. The
equations (\ref{constructa}), (\ref{constructb}) and
(\ref{constructc}) yield a system of three equations with respect
to $a,$ $b$ and $c,$ because $a_{01},$ $a_{02}$ and $a_{03}$ could
be expressed explicitly as functions of $w_{1},$ $w_{2}$ and
$w_{3}$ taking into consideration the exact location of $\triangle
A_{1}A_{2}A_{3}$ which has been given in Theorem~\ref{theor2}.

\end{proof}

\begin{remark}
By replacing the Weirstrass transformations $\sin a=\frac{2
t_{a}}{1+t_{a}^{2}},$ $\cos a=\frac{1-t_{a}^{2}}{1+t_{a}^{2}},$
$\sin b=\frac{2 t_{b}}{1+t_{b}^{2}},$ $\cos
b=\frac{1-t_{b}^{2}}{1+t_{b}^{2}},$ $\sin c=\frac{2
t_{c}}{1+t_{c}^{2}},$ $\cos c=\frac{1-t_{c}^{2}}{1+t_{c}^{2}},$ in
(\ref{constructa}), (\ref{constructb}) and (\ref{constructc}) we
get a system of three rational equations with respect to $t_{a},$
$t_{b}$ and $t_{c}.$ By solving the first derived equation of
second degree with respect to $t_{a},$ we may obtain two solutions
$t_{a_{1}}=f_{1}(t_{b})$ and $t_{a_{2}}=f_{2}(t_{b}).$ Similarly,
by solving the second derived equation of second degree with
respect to $t_{c},$ we may obtain two solutions
$t_{c_{1}}=f_{1}(t_{b})$ and $t_{c_{2}}=f_{2}(t_{b}).$

By replacing these pairs of solutions with respect to
$(t_{a}(t_{b}),t_{c}(t_{b}))$ in the third derived equation we
obtain a rational equation which depend only on $t_{b}.$

\end{remark}

The author is grateful to Professor Dr. Vassilios Papageorgiou for
many fruitful discussions and for his comments on this particular
problem.

\end{document}